\documentclass[11pt,a4paper,reqno]{amsart}
\usepackage{amsmath, amsthm, amscd, amsfonts, amsthm, amssymb, graphicx}
\usepackage[bookmarksnumbered, colorlinks, plainpages]{hyperref}
\headsep=1.2500cm \numberwithin{equation}{section}
\usepackage[]{todonotes}

\newtheorem{theorem}{Theorem}[section]
\newtheorem{lem}[theorem]{Lemma}

\newtheorem{cor}[theorem]{Corollary}

\theoremstyle{definition}

\theoremstyle{notation}

\def\R{\mathop{\mathcal R}}
\def\L{\mathop{\mathcal L}}
\def\A{\mathop{\rm A}}

\begin{document}

\title[On the algebraic structures in $\A_\Phi(G)$]
{On the algebraic structures in $\A_\Phi(G)$}



\author[Ibrahim Akbarbaglu]{Ibrahim Akbarbaglu$^1$}

   \address{$^1$Department of Mathematics, Farhangian University,
  Tehran, Iran}
\email{ibrahim.akbarbaglu@gmail.com, i.akbarbaglu@cfu.ac.ir}

\author[Hasan P. Aghababa]{Hasan P. Aghababa$^2$}
    \address{$^2$Department of Mathematics, University of Tabriz,
  Tabriz, Iran}
\email{pourmahmood@gmail.com, h\_p\_aghababa@tabrizu.ac.ir}

    \address{$^2$School of Computing, University of Utah, Utah, USA}
\email{h.pourmahmoodaghababa@utah.edu}

\author[Hamid Rahkooy]{Hamid Rahkooy$^3$}
    \address{$^3$MPI Informatics, Saarbr\"ucken, Germany}
\email{hrahkooy@mpi-inf.mpg.de}



\maketitle



\begin{abstract}
Let $G$ be a locally compact group and $(\Phi, \Psi)$ be a
complementary pair of $N$-functions. In this paper, using the
powerful tool of porosity, it is proved that when $G$ is an amenable
group, then the Fig\`a-Talamanca-Herz-Orlicz algebra
${\A}_{\Phi}(G)$ is a Banach algebra under convolution product if
and only if $G$ is compact. Then it is shown that ${\A}_{\Phi}(G)$ is a Segal algebra,
and as a consequence, the amenability of ${\A}_{\Phi}(G)$ and the existence of a bounded approximate
identity for ${\A}_{\Phi}(G)$ under the convolution product is discussed. Furthermore, it is shown that for a compact abelian group $G$, the character space of
${\A}_{\Phi}(G)$ under convolution product can be identified with
$\widehat{G}$, the dual of $G$.
\end{abstract}


\section{Introduction}



The Fourier algebra $\A(G)$ of a locally compact group $G$ was
introduced and studied by Eymard in \cite{Eym}. Then Herz in \cite{Her1, Her2}
generalized Fourier algebras into {\it Fig\`a-Talamanca-Herz
  algebras}, denoted by ${\A}_p(G)$ for $1\leq p\leq \infty$. Indeed, it was shown that $\A_2(G)=\A(G)$. The space
$\A_p(G)$ is determined by a homomorphic image of the projective tensor product of the classical Lebesgue spaces $L^p(G)$ and $L^q(G)$, where $q$ is the conjugate exponent to $p$. Fourier and Fig\`a-Talamanca-Herz algebras are well-known concepts in harmonic analysis, have been investigated thoroughly and there is a rich literature on them. For further background, we refer the reader to the valuable books by Pier \cite{Pie} and Derighetti \cite{Der}.

Orlicz spaces are a type of function spaces generalizing the Lebesgue
spaces. A great deal of linear aspects of Orlicz spaces have been
completely characterized in the last few decades.
Rao and Ren's books \cite{Rao1, Rao2}
contain comprehensive information on Orlicz spaces and so
in this paper, we use the definitions and basic facts from these two books.

In \cite{AA}, the first and second authors introduced the
Fig\`a-Talamanca-Herz-Orlicz algebra ${\A}_{\Phi}(G)$ for an $N$-function $\Phi$
and a locally compact group $G$. This is a natural
generalization of Fig\`a-Talamanca-Herz algebras ${\A}_p(G)$ to Orlicz
spaces and it is shown that
${\A}_{\Phi}(G)$ is a Banach algebra under pointwise product of
functions.

In this article, exploring this space more from the algebraic structure point of view, using the concept of porosity,
we address the following question: when is ${\A}_{\Phi}(G)$ a Banach
algebra under the convolution product? Indeed, in Section \ref{sec: conv-prodt},
assuming $G$ is amenable, we show that ${\A}_{\Phi}(G)$ is closed
under the convolution product if and only if $G$ is
compact. As a consequence, we show that ${\A}_{\Phi}(G)$ is a Segal
algebra. The latter statement, in turn, helps us to
characterize amenability of ${\A}_{\Phi}(G)$ and the existence of a
bounded approximate identity in ${\A}_{\Phi}(G)$.
Lastly, we characterize the character space of
${\A}_{\Phi}(G)$ under the convolution product when $G$ is compact and abelian.


\section{Preliminaries}



In this paper, $G$ denotes a locally compact group with a fixed left
Haar measure $\lambda$.  As usual, by $L^0(G)$ we show the set of all
$\lambda$-measurable complex-valued functions on $G$ when $\lambda$-almost
everywhere  equal functions are identified.
The \textit{convolution} product of $f$ and $g$ in $L^0(G)$ is defined
on $G$ by
\[
(f*g)(x)=\int_G f(y)g(y^{-1}x)\,d\lambda(y),
\]
whenever the function $y\mapsto f(y)g(y^{-1}x)$ is Haar integrable for
$\lambda$-almost all $x\in G$.

A locally compact group $G$ is called {\it amenable} if it admits a
left invariant mean \cite{Pie, Run}. It turns out that the amenability
of $G$ can be characterized by {\it Leptin's condition}, that is, for
any $\varepsilon >0$ and compact subset $K\subseteq G$, there exists a
compact subset $U\subseteq G$ such that $0<\lambda(U)< \infty$ and
$\lambda(KU)<(1+\varepsilon)\lambda(U)$ (see \cite[Theorem 7.9,
Proposition 7.11]{Pie}).  For an overview concerning amenability of a
locally compact group see \cite{Pie, Run}.


A convex even function $\Phi:{\mathbb R} \to [0, \infty)$ is called
an {\it $N$-function} if it satisfies $\Phi(0)=0$, $\lim_{x\to 0} \frac{\Phi(x)}{x}= 0$ and $\lim_{x\to \infty} \frac{\Phi(x)}{x}= \infty$.
The convexity of $\Phi$ and $\Phi(0)=0$ implies that $\Phi$ is strictly increasing and so invertible. Indeed, $N$-functions are Young functions with a reasonable behavior \cite{Rao1, Rao2}.

For each $N$-function one can associate another $N$-function
$\Psi:{\mathbb R} \to [0, \infty)$, termed the {\it complementary
  $N$-function} to $\Phi$, defined by
\[
  \Psi(y)=\sup\{x|y|-\Phi(x): x\geq 0\}.
\]
The pair $(\Phi, \Psi)$ is then called the {\it complementary pair of $N$-functions} and satisfies a useful inequality called {\it Young's inequality}:
\begin{equation}\label{Yi}
xy\leq \Phi(x)+\Psi(y) \qquad (x,y\in \mathbb{R}).
\end{equation}
%
%
A number of various $N$-functions, borrowed from \cite{Rao2}, are listed below.
\begin{enumerate}
\item For each $1< p<\infty$, let $\Phi(x)=x^p/p$. Then
  $\Psi(y)=y^q/q$, where $\frac{1}{p}+\frac{1}{q}=1$.
\item If $\Phi(x)= (1+|x|) \log(1+|x|)-|x|$, then
  $\Psi(y)= e^{|y|}-|y|-1$.
\item Let $\Phi(x)=\cosh(x)-1$. Then
  $\Psi(x)=|x|\ln(|x|+\sqrt{1+x^2})-\sqrt{1+x^2}+1$.
\end{enumerate}

According to \cite[Proposition 2.1.1(ii)]{Rao1}, the inverse functions of a
complementary pair of $N$-functions $(\Phi, \Psi)$ satisfy the following inequality:
\begin{equation}\label{12}
x<\Phi^{-1}(x)\Psi^{-1}(x)\leq 2x.
\end{equation}

For each $f\in L^0(G)$ we define
\[
  \rho_\Phi(f)=\int_{G} \Phi\big(|f(x)|\big)\,d\lambda(x).
\]
Given an $N$-function $\Phi$, the {\it Orlicz space} $L^{\Phi}(G)$ is
defined by
\[
  L^{\Phi}(G)= \big\{f\in L^0(G):\, \rho_\Phi(af)< \infty \ \text{for
    some} \ a>0\big\}.
\]
Similarly, the {\it Morse-Transue space} on $G$ is
\[
  M^{\Phi}(G)= \big\{f\in L^0(G):\, \rho_\Phi(af)< \infty \ \text{for
    all} \ a>0\big\}.
\]

The Orlicz space $L^{\Phi}(G)$ is a Banach space under the
Luxemburg-Nakano norm $N_{\Phi}(\cdot)$ defined for any $f\in L^{\Phi}(G)$
by
\[
N_{\Phi}(f)=\inf\{k>0:\, \rho_\Phi(f/k)\leq 1\}.
\]
Also, if $\chi_F$ denotes the characteristic function of a
subset $F\subseteq G$ with $0<\lambda(F)<\infty$, by \cite[corollary 3.4.7]{Rao1} we have
\begin{equation}\label{11}
N_{\Phi}(\chi_F)=\left[\Phi^{-1}\left(\frac{1}{\lambda(F)}\right)\right]^{-1}.
\end{equation}
Moreover, since $\Psi(x)>0$ for each $x \neq 0$, then
\[
  \|f\|_{\Phi}=\sup\bigg\{\int_{G} |fg| d\lambda : \, g \ \text{is
    measurable and}\ N_{\Psi}(g)\leq 1\bigg\},
\]
is another norm called the {\it Orlicz norm} on $L^{\Phi}(G)$.  It
follows from \cite[Proposition 3.3.4]{Rao1} that these two norms are
equivalent and for each $f\in L^{\Phi}(G)$,
\begin{equation}\label{eq6}
N_{\Phi}(f) \leq \|f\|_{\Phi} \leq 2N_{\Phi}(f).
\end{equation}

The space $\A_{\Phi}(G)$ introduced in \cite{AA}, consists of those
$u\in C_0(G)$, the space of all complex-valued continuous functions on
$G$ vanishing at infinity, such that there are sequences
$(f_n)_{n=1}^{\infty}\subseteq M^{\Phi}(G)$ and
$(g_n)_{n=1}^{\infty}\subseteq M^{\Psi}(G)$ with
$\sum_{n=1}^{\infty}N_{\Phi}(f_n)\|g_n\|_{\Psi}<\infty$ and
$u=\sum_{n=1}^{\infty}f_n\ast \check{g}_n$, where
$\check{g}_n(x) = g(x^{-1})$ for $x\in G$. The norm of
$u\in \A_{\Phi}(G)$ is defined by
\[
  \|u\|_{{\A}_{\Phi}}=\inf\left\{\sum_{n=1}^{\infty}N_{\Phi}(f_n)
    \|g_n\|_{\Psi} : u =\sum_{n=1}^{\infty} f_n * \check{g}_n\right\}.
\]
It can be readily seen that $\|u\|_{\infty} \leq \|u\|_{{\A}_{\Phi}}$
(see the proof of \cite[Proposition A.4.5]{Kan} for example).


\section{Convolution product on $\A_{\Phi}(G)$} \label{sec: conv-prodt}


This section is devoted to the clarification of the fact that when
$\A_{\Phi}(G)$ is closed under convolution product. To put in a
nutshell, in the case when $G$ is amenable, we will show that for an
$N$-function $\Phi$, the space $\A_{\Phi}(G)$ is a Banach algebra
under convolution product if and only if $G$ is compact. This
manifests the reason for choosing pointwise product for
$\A_{\Phi}(G)$.

Let us recall the notion of porosity. For a metric space $X$ the open
ball around $x\in X$ of radius $r>0$ is denoted by $B(x, r)$.  Given a
real number $0<c\leq 1$, a subset $M$ of $X$ is called
$c$-\textit{lower porous} if
\[
  \liminf_{R\to 0+}\frac{\gamma(x, M, R)}{R}\geq \frac{c}{2},
\]
for all $x\in M$, where
\[
  \gamma(x, M, R)=\sup\big\{r\geq0: \exists\, z\in X, B(z,r)\subseteq
  B(x,R)\setminus M\big\}.
\]
Clearly $M$ is $c$-lower porous if and only if
\[
  \forall x\in M \ \forall \alpha\in (0, c/2) \ \exists R_0>0 \
  \forall R\in(0, R_0) \ \exists z\in X, \ B(z, \alpha R)\subseteq
  B(x,R)\setminus M.
\]
A set is called $\sigma$-$c$-\textit{lower porous} if it is a
countable union of $c$-lower porous sets with the same constant $c>0$.

It is not hard to observe that a $\sigma$-$c$-lower porous set is
meager and the notion of $\sigma$-porosity is stronger than
meagerness. For more details see \cite{za}.

Let us also remark that for two normed spaces $\mathcal{A}$ and
$\mathcal{B}$, we can endow the spaces $\mathcal{A}\times\mathcal{B}$
and $\mathcal{A}\cap\mathcal{B}$ with the norms
\[
  \|(f,g)\|=\max\{\|f\|_{\mathcal{A}}, \|g\|_{\mathcal{B}}\} \quad
  {\rm and} \quad \|f\|
  =\|f\|_{\mathcal{A}}+\|f\|_{\mathcal{B}},
\]
respectively, for every $f\in\mathcal{A}$ and
$g\in\mathcal{B}$. Notice that we consider the fact that
$\mathcal{A}\cap\mathcal{B}$ is defined for two appropriate normed
spaces $\mathcal{A}$ and $\mathcal{B}$.

The following, taken from \cite{AA}, is a vital lemma in proving the
main result of this section. However, we include its proof here for
convenience and the fact that we will use some of its notation in next
theorems.
\begin{lem}\label{r}\cite{AA}
  Suppose that $G$ is an amenable locally compact group. Let $E$ be a
  compact subset of $G$ and $\Phi$ be an $N$-function. Then for every
  $\epsilon>0$ there exists a function
  $u\in \A_\Phi(G)\cap \A_\Psi(G)$ such that
  $\|u\|_{\A_\Phi}<2(1+\epsilon)$,
  $\|u\|_{\A_\Psi}=\|\check{u}\|_{\A_\Phi}<2(1+\epsilon)$ and $u=1$ on
  $E$.
\end{lem}
\begin{proof}
  Consider a compact subset $F$ of $G$ with $\lambda(F)>0$, and define
  \[
    v(x)=\frac{1}{\lambda(F)}(\chi_{EF}\ast\check{\chi}_F)(x)=\frac{\lambda(xF\cap
      EF)}{\lambda(F)}.
  \]
  Then $v\in {\A}_\Phi(G)$ and $0\leq v \leq 1$. If $x\in E$, then
  $\lambda(xF\cap EF)=\lambda(xF)=\lambda(F)$, so that $v(x)=1$, while
  if $x\notin EFF^{-1}$, then $xF\cap EF=\emptyset$ and therefore
  $v(x)=0$. It follows that $\mbox{supp}\, (v)\subseteq EFF^{-1}$ is
  compact. Now, since $G$ is amenable, by Leptin's condition we may
  select a compact set $V\subseteq G$ of positive measure provided the
  inequality $\lambda(EV)<(1+\epsilon)\lambda(V)$ holds. Now let
  $f=\chi_{EV}$, $g=\frac{1}{\lambda(V)}\chi_V$ and
  $u=f\ast\check{g}$. Then, as $u$ is a special case of $v$, we have
  $u\in \A_\Phi(G)\cap \A_\Psi(G)$ and $u=1$ on $E$. Moreover, using
  \eqref{eq6} we obtain
  \begin{align*}
    \|u\|_{{\A}_\Phi}&\leq\frac{1}{\lambda(V)}\|\chi_V)\|_{\Psi}N_{\Phi}(\chi_{EV})
                       \leq \frac{2}{\lambda(V)}\left[\Phi^{-1}\left(\frac{1}{\lambda(EV)}\right)\right]^{-1}
                       \left[{\Psi}^{-1}\left(\frac{1}{\lambda(V)}\right)\right]^{-1}
                       \vspace{0.1cm} \\
                     & \leq\frac{2}{\lambda(V)}\left[\Phi^{-1}\left(\frac{1}{(1+\epsilon)\lambda(V)}\right)\right]^{-1}
                       \left[{\Psi}^{-1}\left(\frac{1}{(1+\epsilon)\lambda(V)}\right)\right]^{-1}<2(1+\epsilon).
  \end{align*}
  Similarly, we can get $\|u\|_{\A_\Psi}<2(1+\epsilon)$.
\end{proof}
As an immediate consequence of Lemma \ref{r} we can determine when
$\A_\Phi(G)$ is unital.
\begin{cor}\label{cor:unital}
  Let $G$ be a locally compact group and $\Phi$ be an
  $N$-function. Then $\A_\Phi(G)$ is unital if and only if $G$ is
  compact.
\end{cor}
\begin{proof}
  If $G$ is compact, taking $E=G$ in Lemma \ref{r} implies that the
  constant function $1$ belongs to $\A_\Phi(G)$ and so it is
  unital. On the other hand, if $\A_\Phi(G)$ is unital, considering
  the fact that $\A_\Phi(G)$ is equipped with pointwise product, we
  conclude that its unit is the constant function $1$. Since
  $\A_\Phi(G) \subseteq C_0(G)$, we earn $1 \in C_0(G)$. Therefore,
  $G$ is compact.
\end{proof}
\begin{theorem} \label{p} Let $G$ be an amenable locally compact group
  and $\Phi$ be an $N$-function.  If $G$ is not compact and $V$ is a
  symmetric compact neighborhood of the identity, then the set
  \[
    E=\big\{(f,g)\in {\A}_\Phi(G)\times\big({\A}_\Phi(G)\cap
    {\A}_\Psi(G)\big): |f*\check{g}(x)|<\infty, \forall\,x\in V
    \big\},
  \]
  is a $\sigma$-$c$-lower porous set for some $c>0$.
\end{theorem}
\begin{proof}
  Let $V$ be a symmetric compact neighborhood of $e$, the identity of
  $G$.  Since $G$ is not compact, there exists a sequence
  $(a_m)_{m\in\mathbb{N}} \subseteq G$ such that
  $a_iV\cap a_jV=\emptyset$ for distinct natural numbers $i$ and $j$.
  To be quite explicit, let $a_1=e$. Suppose
  $a_1, a_2, \cdots, a_{m-1}$ are chosen.  Since $G$ is not compact,
  there exists $a_m\in G$ so that
  $a_m\notin\bigcup_{i=1}^{m-1}a_i V^2$. Then the collection
  $\{a_m V\}_{m\in\mathbb{N}}$ contains pairwise disjoint subsets of
  $G$. Now for each $n\in\mathbb{N}$ set
  \[
    E_n=\left\{(f,g)\in {\A}_\Phi(G)\times \big({\A}_\Phi(G)\cap
      {\A}_\Psi(G)\big): \int_{G}|f(y)||g(x^{-1}y)|d\lambda(y)\leq n,
      \forall\,x\in V \right\}.
  \]
  Clearly, $E=\bigcup_{n\in {\Bbb N}}E_n$. Hence, we only need to show
  that for each $n\in {\Bbb N}$, $E_n$ is $c$-lower porous for some
  $0<c<1/32$.

  Fix $n\in\mathbb{N}$ and $R>0$ and suppose $(f,g)\in E_n$. Then set
  \[
    A_f^1:=\big\{x\in G: \mbox{Re}f(x)\geq0\big\} \quad \mbox{and}
    \quad A_f^{-1}:=\big\{x\in G: \mbox{Re}f(x)<0\big\}.
  \]
  In the same way we define $A_g^1$ and $A_g^{-1}$. It follows that
  \[
    G=\bigcup_{i,j \in \{1,-1\}}\Big(A_{f}^i\cap A_{g}^j\Big).
  \]
  Then for some function $s:\{1,2\}\to \{1,-1\}$, we can get
  \[
    \lambda\bigg(\bigcup_{m=1}^\infty \Big(a_mV\cap A_f^{s(1)}\cap
    A_g^{s(2)}\Big)\bigg)=\infty.
  \]
  Assume, without loss of generality, that $s(1)=s(2)=1$.  Hence there
  exist $m_0\in\mathbb{N}$ such that
  \[
    \lambda\left(\bigcup_{m=1}^{m_0}\big(a_mV\cap A_f^1\cap
      A_g^1\big)\right)>\frac{512n}{R^2}.
  \]
  Let $K=\bigcup_{m=1}^{m_0} \big(a_mV\cap A_f^1\cap A_g^1\big)$.
  Applying Lemma \ref{r} with $\varepsilon=1$, there exists
  $u\in \A_\Phi(G)\cap {\A}_\Psi(G)$ with $0\leq u\leq 1$, $u=1$ on
  $K$ and $\|u\|_{{\A}_\Phi}+\|u\|_{{\A}_\Psi}\leq 8$. Define
  functions $\widetilde{f}$ and $\widetilde{g}$ on $G$ by setting
  \[
    \widetilde{f}(x):=f(x)+\frac{R}{8}u(x) \quad \mbox{and} \quad
    \widetilde{g}(x):=g(x)+\frac{R}{16}u(x).
  \]
  Now, it can be easily verified that
  $B\big((\widetilde{f},\widetilde{g}),R/32\big)\subset
  B\big((f,g),R\big)$. Therefore, it suffices to show
  $B\big((\widetilde{f},\widetilde{g}),R/32\big)\cap E_n=\emptyset$.
  Fix $(h,k)\in B\big((\widetilde{f},\widetilde{g}),R/32\big)$ and let
  $K_1=\{x\in K:|h(x)|\geq R/16\}$. Then we have
  \begin{align*}
    R/32
    >&
       \|\widetilde{f}-h\|_{{\A}_\Phi}\geq\|\widetilde{f}-h\|_\infty
       \geq
       \sup_{x\in K\setminus K_1}\left\{\big||\widetilde{f}(x)|-|h(x)|\big|\right\}\geq R/16.
  \end{align*}
  However, this is impossible and thus $K_1=K$. Similarly,
  \[
    R/32>\|\widetilde{g}-k\|_{{\A}_\Phi \cap {\A}_{\Psi}} \geq
    \sup_{x\in
      G}\left\{\big||\widetilde{g}(x)|-|k(x)|\big|\right\}\geq
    \sup_{x\in K}\{R/16-|k(x)|\},
  \]
  from which for every $x\in K$ we obtain $|k(x)|>R/32$. On the other
  hand, since $k$ is a continuous function, we can find a symmetric
  compact neighborhood $U$ contained in $V$ such that $|k(x)|>R/32$
  for every $x\in UK$. Now consider an arbitrary element $z\in
  U$. Then one may conclude that
  \begin{align*}
    \int_{K} |h(y)||k(z^{-1}y)|d\lambda(y)\geq\frac{R^2}{512}\lambda(K)>n.
  \end{align*}
  Thus $(h,k)\notin E_n$, as required.
\end{proof}
\begin{cor} \label{cp} Let $G$ be an amenable locally compact group
  and $\Phi$ be an $N$-function.  Then the convolution of each two
  functions in ${\A}_{\Phi}(G)$ exists if and only if $G$ is compact.
  In particular, ${\A}_{\Phi}(G)$ is a Banach algebra under
  convolution product if and only if $G$ is compact.
\end{cor}
\begin{proof}
  Since every Banach space is of second category, the Banach space
  $\A_{\Phi}(G)\times\big({\A}_\Phi(G)\cap {\A}_\Psi(G)\big)$ cannot
  coincide with the set $E$ in Theorem \ref{p} provided $G$ is not
  compact.

  Conversely, let $G$ be compact and $C(G)$ denote the space of
  continuous functions on $G$. Then $C(G)=C_0(G)$ and
  \[
    {\A}_{\Phi}(G) * {\A}_{\Phi}(G) \subseteq C(G) * C(G) \subseteq
    {\A}_{\Phi}(G)
  \]
  whence ${\A}_{\Phi}(G)$ is closed under convolution product.
  Moreover, since $C(G)$ is included in $M^{\Phi}(G)$, it is a Banach
  space with respect to the norm
  $\|f\|_0:=\|f\|_{\infty}+N_{\Phi}(f)$, for $f\in C(G)$. Let
  $I : \big(C(G),\|\cdot\|_0\big) \to
  \big(C(G),\|\cdot\|_{\infty}\big)$ be the identity map, which is
  continuous and one-to-one. By a consequence of the open mapping
  theorem there exists some $\alpha>0$ such that
  $N_{\Phi}(f)\leq \alpha\|f\|_{\infty}$ for all $f\in C(G)$. The same
  argument can be used to derive some $\beta>0$ in order to have
  $\|\check{f}\|_{\Psi}\leq\beta \|f\|_{\infty}$.  Now by the
  definition of $\|\cdot\|_{{\A}_\Phi}$ for $u, v \in {\A}_{\Phi}(G)$
  we have
  \[
    \|u\ast v\|_{{\A}_\Phi}\leq N_{\Phi}(u)\|\check{v}\|_{\Psi}\leq
    \alpha\beta\|u\|_{\infty}\|v\|_{\infty}\leq\alpha\beta\|u\|_{{\A}_\Phi}\|v\|_{{\A}_\Phi}.
  \]
  Therefore $\big({\A}_{\Phi}(G),\|\cdot\|_{{\A}_\Phi}\big)$ is a
  Banach algebra.
\end{proof}

We remark that, unfortunately, we were unable to prove Corollary
\ref{cp} without taking the advantage of amenability property of
$G$. However, we conjecture that this results holds without the
amenability constraint.



We recall that a subspace $S(G)$ of $L^1(G)$ is called a {\it Segal
  algebra} if it satisfies the following four conditions:
\begin{enumerate}
\item $S(G)$ is dense in $L^1(G)$;
\item $\big(S(G), \| \cdot \|_S\big)$ is a Banach space such that
  $\| f\|_1 \leq \|f\|_S$ for all $f \in S(G)$;
\item $\L_t f \in S(G)$ for all $f \in S(G)$ and $t \in G$, where
  $\L_t f (x) = f(t^{-1}x)$ for $x \in G$;
\item For all $f\in S(G)$ the map $\Gamma_f: G \to S(G)$ by
  $\Gamma_f(t) = \L_t f$ is continuous.
\end{enumerate}

A Segal algebra is called {\it symmetric} if
\begin{enumerate}
\item $\R_t f \in S(G)$ for all $f \in S(G)$ and $t \in G$, where
  $\R_t f(x) = f(xt)$ for $x \in G$;
\item For all $f\in S(G)$ the map $\Lambda_f: G \to S(G)$ by
  $\Lambda_f(t) = \R_t f$ is continuous.
\end{enumerate}

According to Proposition 1 of \cite{Joh}, every Segal algebra is an
abstract Segal algebra (see \cite{Joh, Bur} for its definition) with
respect to $L^1(G)$.

As an immediate consequence of Corollary \ref{cp}, we obtain the following.
\begin{theorem} \label{thm: Sym-Segal} Let $G$ be a compact group with
  normalized left Haar measure and $\Phi$ be an $N$-function.  Then
  $\big({\A}_{\Phi}(G),\|\cdot\|_{{\A}_\Phi}\big)$ is a symmetric
  Segal algebra with respect to $L^1(G)$. In particular,
  $\big({\A}_{\Phi}(G), \|\cdot\|_{{\A}_\Phi}\big)$ is an abstract
  Segal algebra with respect to $L^1(G)$.
\end{theorem}
\begin{proof}
  By Lemma \ref{r}, $\A_{\Phi}(G)$ strongly separates the points of
  $L^1(G)$. Therefore, $\A_{\Phi}(G)$ is dense in $L^1(G)$ by the
  Stone-Weierstrass theorem. Also, since $\lambda(G)=1$, for any
  $f\in\A_{\Phi}(G)$ we have
  $\|f\|_1\leq\|f\|_{\infty}\leq\|f\|_{\A_\Phi}$.

  To complete the proof we note that for any Young function $\Phi$,
  $M^{\Phi}(G)$ is left translation invariant and
  $\|\L_t h\|_{\Phi} = \|h\|_{\Phi}$ whenever $h\in M^{\Phi}(G)$.  Now
  take $u\in\A_{\Phi}(G)$. For simplicity assume $u = f * \check{g}$,
  where $f \in M^{\Phi}(G)$ and $g \in M^{\Psi}(G)$.  Then for any
  $t \in G$, it can be readily observed that
  $\L_t u=(\L_t f) * \check{g}$ and $\R_t u = f * (\L_t g{\check{)}}$.
  This yield that $\L_t u, \R_t u \in \A_{\Phi}(G)$ and
  $\|\L_t u\|_{\A_{\Phi}} =\|\R_t u\|_{\A_{\Phi}} =
  \|u\|_{\A_{\Phi}}$.  Finally, by Proposition 5.3 in \cite{AA} the
  mappings $t \to \L_t f$ and $t \to \R_t f$ from $G$ to
  $\A_{\Phi}(G)$ are continuous.
\end{proof}

It turns out that a symmetric Segal algebra is indeed a two-sided
ideal in $L^1(G)$ and has a $L^1$-bounded approximate identity (see
\cite{Rei1, Rei2}). The followings are some consequences of the
preceding theorem.  For the notion of amenability of a Banach algebra
see \cite{Run}.

\begin{cor}\label{cor:am}
Let $G$ be a compact group, $\Phi$ be an $N$-function and consider $\A_{\Phi}(G)$ as a Banach algebra under convolution product. Then the following are equivalent:
\begin{itemize}
\item[(i)] $\A_{\Phi}(G)$ has a bounded approximate identity.
\item[(ii)] $\A_{\Phi}(G)$ is unital.
\item[(iii)] $\A_{\Phi}(G)$ is amenable.
\item[(iv)] $G$ is finite.
\end{itemize}
\end{cor}
\begin{proof}
As $\A_{\Phi}(G)$ is a Segal algebra by Theorem \ref{thm:
Sym-Segal}, according to \cite[Theorem 1.2]{Bur}, it cannot have a
$\|\cdot\|_{\A_{\Phi}}$-bounded approximate identity unless
$\A_{\Phi}(G) = L^1(G)$. Since $\A_{\Phi}(G) \subseteq C(G) \subseteq L^1(G)$, we conclude that $C(G) = L^1(G)$ and so $G$ is finite. This also
is equivalent to the existence of a unit for $\A_{\Phi}(G)$. As
amenability implies the existence of a bounded approximate identity
by \cite[Proposition 2.2.1]{Run}, (iv) is equivalent to others too.
\end{proof}

We remark that Corollaries \ref{cor:unital} and \ref{cor:am} show that
how algebraic structures can make really different spaces.

The {\it character space} of a Banach algebra $A$ is defined to be the
set of all bounded multiplicative linear functionals on $A$ and is
denoted by $\Delta(A)$.  As another application of Theorem \ref{thm: Sym-Segal}, along with \cite[Theorem 2.1]{Bur}, we present the following result about compact
 abelian groups.

\begin{cor}
Let $G$ be a compact  abelian group and $\Phi$ be an
$N$-function. Then $\Delta\big(\A_\Phi(G)\big)$ is homeomorphic to
$\widehat{G}$, the dual of $G$, where $\A_\Phi(G)$ is considered with convolution product.
\end{cor}

\subsection*{Remark.} 

In the first draft of the paper it was shown that the character space of $\A_{\Phi}(G)$ under pointwise product coincides with $G$, for an arbitrary locally compact group $G$ and an $N$-function $\Phi$. It is, however, pointed out by the reviewer that it was already given in \cite{LK} (see Theorems 3.6(i), 3.7 and Corollary 3.8). Therefore, we preferred to do not include them in the final version of the paper. However, it is worth mentioning that our approach was taken from \cite{Kan} (specially Theorem 2.9.4), which was a bit different from the approach given in \cite{LK}.

\providecommand{\bysame}{\leavevmode\hbox to3em{\hrulefill}\thinspace}


\end{document}